\documentclass[psamsfonts]{amsart}

\usepackage{cite}
\usepackage{eucal}
\usepackage{amsopn}
\usepackage{amsmath}
\usepackage{amsfonts}
\usepackage{amssymb}
\usepackage{amsthm}

\newtheorem{theorem}{Theorem}[section]
\newtheorem{lemma}[theorem]{Lemma}

\theoremstyle{definition}
\newtheorem{definition}[theorem]{Definition}

\theoremstyle{remark}

\numberwithin{equation}{section}

\begin{document}
\title{A Rigidity Condition for Generators in Strongly Convex
Domains}
\author{Tiziano Casavecchia}
\address{Mathemathics Department L. Tonelli of Pisa University, Largo B.
Pontecorvo 5, Pisa, 56127, Italy}
\email{casavecc@mail.dm.unipi.it}
\date{14-OCTOBER-2008}
\subjclass[2000]{Primary 32; Secondary 37}
\begin{abstract}
Let \(F\) be an infinitesimal generator of a semigroup of holomorphic self-maps
in a smooth strongly convex subdomain \(D\) of \(\mathbb{C}^{n}\). We prove
that \(F\equiv 0\) on \(D\) if \(F\) vanishes in angualr sense at a boundary
point up to third order.
\end{abstract}
\maketitle

A \textit{semigroup} \(\Phi_{t}\) of holomorphic maps in a subdomain \(D\) of
\(\mathbb{C}^{n}\) is a continuous homomorphism from the additive semigroup
\((\mathbb{R}^{+},+)\) into the semigroup \textit{Hol}(\(D\),\(D\)) of
holomorphic
self-maps of \(D\), respect to the operation of composition, endowed with
the
topology of uniform convergence on compacts subsets.
We know (\cite[Section ~2.5.3]{abatelibro}, \cite{shoikhetlibro} and
\cite{shoikhetlibro2}) that the function \([0,+\infty) \ni t \mapsto
\Phi_{t}\in \text{\textit{Hol}}(D,D)\) is analytic, and that to each such
semigroup
there corresponds a vector field \(F:D \rightarrow \mathbb{C}^{n}\) (as usual we
identify \(\mathbb{C}^{n}\) with its tangent space),
such that \(\frac{\partial \Phi_{t}}{\partial t}=F(\Phi_{t})\) (It should be
noted
that the book \cite{shoikhetlibro} uses a different sign convention, so some
formulas may appear a bit different). This vector field is usually called the
\textit{infinitesimal} \textit{generator} of the semigroup. It is a
\textit{semicomplete} vector field, in the sense that each maximal solution
\(\gamma _{z}\), with \(\gamma (0)=z\), can be extended up to \(+\infty\).
On the other hand, let \(D\) be a subdomain of \(\mathbb{C}^{n}\) and
\(F:D\rightarrow \mathbb{C}^{n}\) a
holomorphic map; if, for each
\(z\in D\), the Cauchy problem given by
\begin{equation*}\begin{cases}\dfrac{\partial\Phi_{z}(t)}{\partial
t}=F(\Phi_{z}(t))\\
\Phi_{z}(0)=z\\ \end{cases}\end{equation*}
has a unique maximal solution defined on \([0,+\infty)\), then \(F\) is
the infinitesimal generator of a unique one-parameter semigroup of holomorphic
self-maps of \(D\).

Not all holomorphic maps defined on a domain \(D\) are infinitesimal
generators of some semigroups of holomorphic self-maps of \(D\); for
\(D=\Delta\) (the open unit disc of \(\mathbb{C}\)), we have the following
powerful representation formula found by Berkson and Porta
(\cite{berksonporta} and the books \cite{abatelibro}, \cite{shoikhetlibro}
and \cite{shoikhetlibro2}):
\begin{theorem}[Berkson-Porta]
Let \(f:\Delta\rightarrow\mathbb{C}\) be a holomorphic function. Then \(f\) is
the infinitesimal generator of a semigroup of holomorphic self-maps of
\(\Delta\) if and only if there exists \( b\in\overline{\Delta}\) and a
holomorphic function \(p:\Delta \rightarrow \mathbb{C}\), such that \(\Re p
\geq
0\) and \begin{equation}\label{bp}
f(\xi)=(\xi - b)(\overline{b} \xi - 1)p(\xi).
\end{equation}
Furthermore \(b\) and \(p:\Delta\rightarrow\mathbb{C}\) are uniquely determined
by \(f\).
\end{theorem}

For us a \textit{rigidity condition} is a sufficient condition on \(F\)
ensuring \(F\) has to vanish on \(D\). Since the vanishing of \(F\) on \(D\)
is
equivalent to the semigroup generated by \(F\) being trivially the identity map on \(D\), a rigidity condition
on \(F\) may be linked to some condition on a self map of \(D\) that forces it to be the identity map.
In this line of arguments D. M. Burns and S. G. Krantz published an article
(\cite{burnskrantz}) providing a condition under which a holomorphic map \(
\phi :\Delta\rightarrow\Delta\) is the identity:
\begin{theorem}
Let \(\phi :\Delta\rightarrow\Delta\) be a holomorphic map and \(\tau\in
\partial\Delta\). If \begin{equation*}
\phi (\xi)= \tau +(\xi -\tau)+O(\lvert \xi -\tau\rvert ^4)
\end{equation*}
for \(\xi\rightarrow 1\), then \(\phi \equiv id_{\Delta}\).
\end{theorem}
Furthermore they extend this result to \(\mathbb{B}^{n}\), the open unit ball of
\(\mathbb{C}^{n}\) and even to more general domains; this theorem is a
kind of Schwarz lemma at the boundary.

In recent years R. Tauraso and F. Vlacci (\cite{taurasovlacci}) and F. Bracci,
R. Tauraso and F. Vlacci (\cite{bracci4}) improved this result showing that
the unrestricted limit can be replaced by an angular one or, being equivalent in
\(\Delta\), by the limit along a non tangential curve. They used this to prove,
in the last paper cited above, an identity principle for commuting self-maps
of
\(\Delta\). Further generalizations of these results to more general situations
can be found for instance in \cite{huang}, \cite{baracco} and \cite{migliorini}.

As first discovered by A. Kor\'anyi and E. M. Stein (\cite{koranyi} and
\cite{koranyistein}), the right (for function theory) generalization to the unit ball in several
complex variables of angular regions in \(\Delta\) is not a conic region, but
the following
set:\begin{equation*}K(\tau,M)=\{z\in\mathbb{B}^{n}\mid\frac{\lvert
1-(z,\tau)\rvert}{1-\lVert z\rVert} « M\},\end{equation*}where \(\tau\) is a
point in the boundary of \(\mathbb{B}^{n}\), \((\cdot,\cdot)\) is the
standard hermitian product in \(\mathbb{C}^{n}\) and \(M »1\). These regions are
called \textit{Kor\'anyi regions of center \(\tau\) and amplitude \(M\)}. In
\(\Delta\) these regions are just egg-shape sectors with a corner at
\(\tau\), symmetric with respect to the line segment from \(0\) to \(\tau\). In
\(\mathbb{B}^{n}\) they are
angular only along the direction normal to the boundary of \(\mathbb{B}^{n}\)
and tangential in all others directions. Kor\'anyi regions can be introduced
also in Hilbert spaces. 
Using Kor\'anyi regions we can define \(K\)-\textit{limit} at any point
\(\tau\)
in the boundary of \(D\); we say that \(F:D\rightarrow \mathbb{C}^{n}\) has
\(K\)-limit \(L\) and we write \begin{equation*}K\text{-}\lim
_{z\rightarrow \tau}F(z)=L\end{equation*} if \(\lim F(z)=L\) for \(z\rightarrow
\tau\)
inside the Kor\'anyi region \(K(\tau,M)\) with center \(\tau\), for any \(M
»1\). In \(\Delta\) for a function to have \(K\)-limit, is the same that to
have angular limit. 

In 2007 M. Elin, M. Levenshtein, S. Reich, D. Shoikhet (\cite{shoikhet1})
proved the following couple of theorems:
\begin{theorem}\label{shoresult1a}Let \(f:\Delta
\rightarrow \mathbb{C}\) be the infinitesimal generator
of a semigroup of holomorphic self-maps of \(\Delta\). If
\begin{equation}\label{eqshoikhetresult1a}
K\text{-}\lim_{\xi \rightarrow \tau}\frac{f(\xi)}{\vert \xi-1 \vert^{3}} =
0,
\end{equation}then \(f\equiv 0\) on \(\Delta\).
\end{theorem}
\begin{theorem}\label{shoresult}Let
\(\mathbb{B}\) be the unit ball in a
Hilbert
space \textup{\textbf{H}}, \(\tau \in\partial\mathbb{B}\)
and \(F:\mathbb{B} \rightarrow \textup{\textbf{H}}\) the infinitesimal generator
of a semigroup of holomorphic self-maps of \(\mathbb{B}\). If
\begin{equation}\label{eqshoikhetresult}
K\text{-}\lim_{z \rightarrow \tau}\frac{F(z)}{\lVert z-\tau \rVert^{3}} =
0,
\end{equation}then \(F\equiv 0\).
\end{theorem}
They prove the former using dynamics properties of \(f\) and unicity of
\(b\) and \(p:\Delta\rightarrow\mathbb{C}\) in \eqref{bp}.
Their proof of the latter uses automorphisms of the unit ball in a
Hilbert space to reduce to the one dimensional case.
So, on one side it is clear that a weaker limit than the unrestricted one is
sufficient to guarantee a rigidity condition, on the other the intervention
of Kobayashi distance in the ball suggests possible generalizations involving
the
Kobayashi distance in more general domains.

In \cite{abatederivateangolari1}, Abate found a generalization of Kor\'anyi
region to any smooth strictly convex domain \(D\) of \(\mathbb{C}^{n}\). That
is, given \(\tau\) in the boundary of \(D\) and \(x\) in \(D\),
a \textit{Kor\'anyi region of center \(\tau\), pole \(x\) and amplitude \(M\)}
is the set of points \begin{equation*}K^{D}(\tau ,p,M)=\{z\in D\mid
\lim_{w\rightarrow \tau}[k_{D}(z,w)-k_{D}(p,w)]+k_{D}(p,z)«\log
M\}.\end{equation*} Here \(k_{D}\) is the Kobayashi distance in \(D\). In
\cite{abateorosfere} Abate proved that the previous limit exists and in
\cite{abatederivateangolari1} he proved that \(K^{\mathbb{B}^n}(\tau ,0,M)\)
and \(K(\tau,M)\) are the same set (see also \cite[Chapter ~2.7]{abatelibro}
and \cite{abatederivateangolari2}).
We now recall the notions
needed to understand the statement of Theorem \ref{maintheorem} and refer to
section 1 for more details and bibliography.

Again, \(D\) denotes a smooth domain of \(\mathbb{C}^{n}\) and \(\tau\) a point in the \(\partial D\), the boundary
of \(D\). We call a continuous curve \(\alpha :[0,1)\rightarrow D\) a
\textit{\(\tau\)-curve} if \(\lim_{t\rightarrow 1^{-}}\alpha (t)=\tau\). A
\textit{conic region} at 
\(\tau\) of amplitude \(c\) is the set of points \begin{equation*}\{z\in D\mid
\lVert
z-\tau\rVert\leq c\, \text{dist}(z,\partial D)\}\end{equation*} for some
\(c » 1\) where \(\text{dist}(\cdot,\cdot)\) denotes the
euclidean distance in \(\mathbb{C}^{n}\).
A
\textit{non-tangential} \(\tau\)-curve is a \(\tau\)-curve that
lies inside some conic region at \(\tau\).
Given a function \(F:D\rightarrow \mathbb{C}^{n}\), we say that \(F\) has
\textit{non-tangential limit} \(L\) at
\(\tau\) and we write \begin{equation*}\angle\lim_{z\rightarrow
\tau}F(z)=L\end{equation*}
if \(\lim_{t\rightarrow 1^{-}}F(\alpha (t))=L\) along any non-tangential
\(\tau\)-curve. For comprehensive general reference on
non-tangential limits we refer
to \cite{abatelibro},
\cite{abatederivateangolari1}, \cite{abatederivateangolari2} and \cite{rudin}.

The main
theorem of this
paper is the following.\begin{theorem}\label{maintheorem}
Let D be a bounded smooth strongly convex domain of \(\mathbb{C}^n\);
\(F:D\rightarrow
\mathbb{C}^n\) be the infinitesimal generator of a semigroup of
holomorphic self-maps of D and
\(\tau \in \partial D\). If
\begin{equation*}
\angle\lim_{z \rightarrow \tau}\frac{F(z)}{\lVert z-\tau
\rVert^{3}} =
0 ,
\end{equation*} then \(F \equiv 0\) in \( D\).
\end{theorem} 

Theorem \ref{maintheorem} is stronger than Theorem \ref{shoresult} even in
the ball, since, as
we shall explain better in the next section, \(K\)-limits are stronger than
non-tangential ones, since each non-tangential \(\tau\)-curve
is eventually contained in some Kor\'anyi region at \(\tau\).
We furtermore remark that the last theorem holds also if \(D\) has boundary of
class \(\mathcal{C}^{k}\), \(k\geq 14\) and in bounded strictly
linearly convex domain with the same regularity at the boundary.

This paper is organized as follows. In section \ref{sezione1}
we shall recall some standard tools
needed for the proof, mainly the Lempert theory of complex geodesics. In
section 2 we will prove, following \cite{bracci3} some results allowing us
to use Theorem \ref{shoresult} about the infinitesimal generators of continuous
semigroups of holomorphic self-maps of the unit disc \( \Delta
\subseteq \mathbb{C}\). In section \ref{sezione3} we shall prove our main
result.

Finally I would like to thank mainly Marco Abate, my research director, 
for guiding me in studying this topic and problem, and Filippo Bracci, Simeon
Reich for their suggestions. 

\section{General Framework}\label{sezione1}

Let \(D\) be a bounded, smooth, strongly convex domain of
\(\mathbb{C}^n\) and denote the Kobayashi
distance in \(D\) by \(k_{D}\). A \textit{complex geodesic} is
a holomorphic map \(\varphi :\Delta \rightarrow D\) which is an isometry
beetween \(k_{\Delta}= \omega\), the Poincar\'e distance in \(\Delta\), and
\(k_{D}\). Any complex geodesic extends smoothly to the boundary, maps
\(\partial \Delta\) in \( \partial D\) and further for each \(\xi\) in
\(\partial\Delta\), \(\varphi^{\prime}(\xi)\) is transversal to the boundary of
\(D\).

To each complex geodesic \(\varphi\) is associated a "\textit{dual map}" (see
\cite[Chapter ~2.6]{abatelibro}),
\(\tilde{\varphi}:\overline{\Delta}\rightarrow \mathbb{C}^{n}\), holomorphic in
\(\Delta\), smooth up to the boundary, such that
\(\tilde{\varphi}(\xi)=\xi h(\xi)\partial r(\varphi(\xi))\) where \(\xi\) is in
\(\partial \Delta\) and \(h:\partial \Delta\rightarrow\mathbb{R}\)
is a smooth and positive function. The map \(\tilde{\varphi}\) is determined up
to a positive
constant which we normalize by requiring that \(\langle\varphi^{\prime},
\tilde{\varphi}\rangle\equiv 1\) where \(\langle \cdot,\cdot \rangle\) denotes
the
standard bilinear (not Hermitian!) form in \(\mathbb{C}^{n}\). 

Using
\(\tilde{\varphi}\) we can define a left inverse for \(\varphi\): the
holomorphic map such that for each \(z\) in \(D\), there is only one \(\xi \in
\Delta\) which solves the equation \(\langle
z-\varphi(\xi),\tilde{\varphi}(\xi)\rangle=0\). We call
\(\tilde{\rho}_{\varphi}:D\rightarrow \Delta\) the map obtained in this
way. It can be proved that \(\tilde{\rho}_{\varphi}\circ\varphi\equiv
\textup{\textit{id}}\) in
\(\Delta\). The map \(\tilde{\rho}_{\varphi}\) even
extends smoothly up to the boundary of \(D\) and the fibers of
\(\tilde{\rho}_{\varphi}\) are intersections of \(D\) with complex hyperplanes
.

We put
\(\rho_{\varphi}=\varphi\circ\tilde{\rho}_{\varphi}\); then
\(\rho_{\varphi}\circ\rho_{\varphi}=\rho_{\varphi}\) and
\(\rho_{\varphi}\circ\iota_{\varphi(\Delta)}=id_{\varphi(\Delta)}\). The
map \(\rho_{\varphi}\) extends smoothly up to the boundary and
\(\tilde{\rho}_{\varphi}(\overline{D}\setminus\varphi(\partial\Delta))\subseteq
\Delta\). From the definition of \(\varphi\), being it injective, it follows that the
fibers of \(\rho_{\varphi}\) are also intersections with \(D\) of complex affine
hyperplanes. We call \((\varphi,\rho_{\varphi},\tilde{\rho}_{\varphi})\)
the \textit{projection device} associated to the complex geodesic \(\varphi\).

One of the
consequences of Lempert's work is that the
Kobayashi distance \(k_{D}\) is smooth outside the diagonal of \(D\times D\).
Furthermore for any couple of points \(z\), \(w\) in \(\overline{D}\) there
exists just one complex geodesic \(\varphi\), modulo precomposition by
automorphisms of \(\Delta\), such that \(\{z,w\}\subseteq
\varphi(\overline{\Delta})\). For complete statements and proofs we refer to
\cite[Chapter ~2.6]{abatelibro}, \cite[Chapter ~5]{kobayashilibro2} and
\cite[Chapter ~4]{jarnickilibro} and, to the original works
\cite{lempert1}, \cite{lempert2}.

In studying global iteration theory in convex domains, M.
Abate introduced the concept of horosphere in the frame of invariant distances
and this notion proves to be very useful. Let \(D\) be as above, \(\tau\) a
point in \(\partial D\) and \(p\) a point in \(D\). The
horosphere of center \(\tau\), pole \(p\in D\) and radius \(R\) is defined as
the set \begin{equation*}E(\tau,p,R)=\{z\in D\mid \lim_{x\rightarrow
\tau}[k_{D}(z,x)-k_{D}(p,x)]«\dfrac{1}{2}\log R\}.
\end{equation*} It was proved in \cite{abateorosfere} that the limit exists,
and horospheres are convex subsets of \(D\) (see also
\cite[Theorem ~2.6.47]{abatelibro}). In the unit disc of
\(\mathbb{C}\) they are discs internally tangents to the boundary of the unit
disc in the center of the horosphere. In the unit ball of
\(\mathbb{C}^{n}\) they are ellipsoids internally tangents to the boundary of
the unit ball in the center of the horosphere.
Recently it has been showed
(\cite{trapani}, \cite[Section ~4]{bracci2} and \cite[Remark ~6.4]{bracci1})
that horospheres in bounded smooth strongly convex domains are smooth and
strongly convex (and thus strongly pseudoconvex), except at the center and have
a global defining function (\cite[Theorem ~6.3]{bracci1}). Precisely they are
sublevel sets of the pluricomplex Poisson kernel in \(D\) (see \cite{bracci2}).

The last tool we need is restricted limit. In many problems where the
central feature is the boundary behavior of holomorphic maps, they prove to be
very useful. We recall the basic definitions and proprieties, previously stated
in the introduction. Let \(D\),
\(\tau\) and \(p\) be as above. As we said \textit{conic region at} 
\(\tau\) is the set of points \begin{equation*}\{z\in D\mid \lVert
z-\tau\rVert\leq c\, \text{dist}(z,\partial D)\}\end{equation*} for some \(c »
1\), where \(\text{dist}(\cdot,\cdot)\) denotes the euclidean distance in
\(\mathbb{C}^{n}\). A \textit{Kor\'anyi region} centered at
\(\tau\), with pole \(p\) and amplitude \(M »1\) is the set
\begin{equation*}K(\tau,p,M)=\{z\in D\mid
\lim_{x\rightarrow\tau}[k_{D}(z,x)-k_{D}(p,x)] +k_{D}(z,p)«\log
M\}.\end{equation*}
Recall again that a \(\tau\)\textit{-curve} is a map \(\alpha :[0,1)\rightarrow
D\)
such that \(\lim_{t\rightarrow 1^{-}}\alpha (t)=\tau\). A
\textit{non-tangential}
\(\tau\)-curve is a \(\tau\)-curve that lies inside some conic region at
\(\tau\).
The prototype of a non-tangential \(\tau\)-curve is \(\varphi
(r)\) for
\(r\) in \([0,1)\), where \(\varphi\) is a complex geodesics such that
\(\varphi (1)=\tau\).

A
function \(F:D\rightarrow \mathbb{C}^{n}\) has \textit{non-tangential
limit} \(L\) at \(\tau\) and we
denote it by \begin{equation*}\angle\lim_{z\rightarrow
\tau}F(z)=L\end{equation*} if \(\lim_{t\rightarrow
1^{-}}F(\alpha (t))\) exists and is the same for any
non-tangential \(\tau\)-curve. 

We say that \(F:D\rightarrow \mathbb{C^{n}}\) has
\(K\text{-}\lim\) \(L\) and we write \begin{equation*}K\text{-}\lim
_{z\rightarrow \tau}F(z)=L\end{equation*} if \(\lim F(z)=L\) for \(z\rightarrow
\tau\)
whitin \(K(\tau,M)\), for any \(M
»1\).

The relation between these limits is the following: if \(K\text{-}\lim
_{z\rightarrow \tau}F(z)=L\)
then \(\angle\lim_{z\rightarrow
\tau}F(z)=L\). This implication is generally strict.

For all the latter definitions and statements we refer the reader to
\cite[Chapter
~2.7]{abatelibro}, \cite{abatederivateangolari1}
and \cite{abatederivateangolari2} for the general case and to \cite{rudin} for
the unit ball.

\section{Infinitesimal Generator}\label{sezione2}
We start this section with a definition that will simplify the following
statements.\begin{definition}Let \begin{equation*}\begin{split}
\textit{HolG}(D)=&\{F \in
\textit{Hol}(D,\mathbb{C}^{n}) \mid \text{F is the infinitesimal generator of a
one} \\
& \text{parameter semigroup of holomorphic self-maps of D}\}. \end{split}
\end{equation*}
\end{definition}

A recent paper (\cite{bracci3}) provides us with the following useful condition
under which a map \(F\in \text{\textit{Hol}}(D,\mathbb{C}^n)\) is in
\(\text{\textit{HolG}}(D)\), namely:
\begin{theorem}\textup{\cite[Theorem ~3.5]{bracci3}}\label{prop1} Let \(D\)
be a smooth strongly convex domain of \(\mathbb{C}^{n}\), \(F\in
\text{\textit{Hol}}(D,\mathbb{C}^n)\); then \(F \in
\text{HolG}(D)\) if and only if \(\forall z,w \in D\), \(z\neq w\), we have
\begin{equation*} (dk_{D})_{(z,w)}(F(z),F(w))\leq 0.\end{equation*}
\end{theorem}
This property of \(F\) says more or less that the flow generated goes inside
each ball in the Kobayashi distance. The proof uses a few very basic notions of
potential theory but can quite easily
be adapted to avoid them. 

A similar condition with the Kobayashi metric in
place of the Kobayashi distance was found in \cite{abategeneratore}. Both
of these are related to the notion of \(k_{D}\)-\textit{monotonicity} as
introduced, for instance, in \cite{shoikhetlibro2}.

The following couple of results are implicit in the proof of
\cite[Proposition ~4.5]{bracci3}. In order to make this paper self contained we
shall state and prove them here without using potential theoretic notions.
\begin{lemma}\label{lemma111} Let \(z\in D\), let \(B_{k_{D}}(z,R)\) be
a Kobayashi-ball of \(D\) centered in \(z\) and with radius \(R\), and
\(
w\in\partial B_{k_{D}}(z,R)\); furthermore let
\((\rho_{\varphi},\tilde{\rho}_{\varphi},\varphi)\) be a projection device such
that \(z,w\in\varphi(\Delta)\). Then \begin{equation*}
T^{\mathbb{C}}_{w} \partial B_{k_{D}}(z,R)=\ker (d\rho_{\varphi})_{w}.
\end{equation*} \end{lemma}
\begin{proof} Observe that
\(d\rho_{\varphi}=\partial\rho_{\varphi}\) because \( \rho_{\varphi} \) is
holomorphic.
The fibers of \(\rho_{\varphi}\) are intersection with \(D\) of complex affine
hyperplanes and we have
\(\ker
(d\rho_{\varphi})_{w}=\ker(\partial\rho_{\varphi})_{w}=\{y-w \mid y\in
\rho_{\varphi}^{-1}(w)\}\). Now
\(\rho_{\varphi}(z)=z\), \(\rho_{\varphi}(w)=w\) because \(\rho_{\varphi}\) is a
retraction on \(\varphi(\Delta)\); then
\(\rho_{\varphi}(B_{k_{D}}(z,R))\subseteq
B_{k_{D}}(z,R)\) because \(\rho_{\varphi}\) is holomorphic and contracts the
Kobayashi distance. So we have \( \rho_{\varphi}(B_{k_{D}}(z,R)) \subseteq
B_{k_{D}}(z,R)\cap \varphi(\Delta) \), thus no point
\(y\in\rho_{\varphi}^{-1}(w)\) can lie inside \( B_{k_{D}}(z,R)\). We know that
the
balls in the Kobayashi distance in \(D\) are
convex subsets with smooth boundary, so each point in the boundary has a unique
real
tangent hyperplane and also a unique complex tangent hyperplane. Then the
complex hyperplane \(\{y\in \mathbb{C}^{n} \mid \rho_{\varphi}(y)=w\}\) is the
complex tangent hyperplane at \(\partial B_{k_{D}}(z,R)\) in \(w\) and thus
\begin{equation*}
T^{\mathbb{C}}_{w} \partial B_{k_{D}}(z,R)=\ker (d\rho_{\varphi})_{w}.
\end{equation*}
\end{proof}

\begin{lemma}\label{lemma1a} For any complex geodesic \(\varphi:\Delta
\rightarrow D\), with associated
\((\rho_{\varphi},\tilde{\rho}_{\varphi},\varphi)\) projection device and for
any \(z=\varphi(\xi)\in \varphi(\Delta)\), \(w=\varphi(\eta) \in \varphi
(\Delta)\), with \(z\neq w\) and vectors \(u \in T^{\mathbb{C}}_{z}(D)\),
\(v \in T^{\mathbb{C}}_{w}(D)\) we have \begin{equation*}
(dk_{D})_{(z,w)}(u,v)=(dk_{\Delta})_{(\xi,\eta)}((d\tilde{\rho}_{\varphi})_{z}
(u) , (d\tilde{\rho}_{\varphi})_{w}(v)).
\end{equation*}
\end{lemma}
\begin{proof} First of all, from the fact that
\(\varphi\) is an isometry, we have,
\(k_{D}(\varphi(\delta),\varphi(\theta))=k_{\Delta}(\delta,\theta)\), for all \(
\delta, \theta \in \Delta\), and hence
\begin{equation*}\begin{split} d(k_{D})_{(\varphi(\xi),
\varphi(\eta))}[(d\varphi)_ { \xi
}(\zeta),(d\varphi)_{\eta}(\sigma)]& =d_{x}(k_{D}(x,y))_{(\varphi(\xi),
\varphi(\eta)) }
[(d\varphi)_ { \xi }(\zeta),0] \\
& \quad + d_ { y }(k_ { D }(x ,
y))_{(\varphi(\xi),\varphi(\eta))}[0,(d\varphi)_{\eta}(\sigma)]\\
& =
(dk_{\Delta})_{(\xi,\eta)}(\zeta,\sigma)\end{split}\end{equation*}for any
\(\zeta\) in
\(T_{\xi}(\mathbb{C})\cong\mathbb{C}\) and \(\sigma\)
in \(T_{\eta}(\mathbb{C})\cong\mathbb{C}\).
We claim that
\begin{equation}\label{eq1a}
d_{x}(k_{D}(x,y))_{(z,w)}(u,0)=d_{x}(k_{D}(x,y))_{(z,w)}[(d\rho_{
\varphi})_{z}(u),0]
\end{equation} and\begin{equation}\label{eq1b}
d_{y}(k_{D}(x,y))_{(z,w)}(0,v)=d_{y}(k_{D}(x,y))_{(z,w)}[0,(d\rho_{
\varphi})_{w}(v)],
\end{equation} for any \((u,v)\) in \(T_{z}^{\mathbb{C}}(D)\times
T_{w}^{\mathbb{C}}(D)\). We shall prove only (\ref{eq1b}), the other proof is
similar. Since \(\rho_{\varphi}\) is a
holomorphic retraction on \(\varphi(\Delta)\), we have
\(d\rho_{\varphi}=\partial\rho_{\varphi}\) and \((d\rho_{\varphi}) _{w}\circ
(d\rho_{\varphi})_{w}= (d\rho_{\varphi})_{w}\); thus \(d\rho_{\varphi}\) is a
linear projection in \(T_{w}^{\mathbb{C}}(D)\) and we have a holomorphic
splitting \begin{equation*}\label{eq2a}
T_{w}^{\mathbb{C}}(D)=d\rho_{\varphi}(T_{w}^{\mathbb{C}}(D)) \oplus \ker
(d\rho_{\varphi})_{w}.
\end{equation*}Using the previous lemma we have
\begin{equation} \ker (d\rho_{\varphi})_{w}= T^{\mathbb{C}}_{w} \partial
B_{k_{D}}(z,R)=\ker d_{y}(k_{D}(x,y))_{(z,w)} \end{equation}
where \(R=k_{D}(z,w)\). This concludes the proof of the claim.
Hence we have \begin{equation*} \begin{split}
(dk_{D})_{(z,w)}(u,v)& = d_{x}(k_{D}(x,y))_{(z,w)}(u,0)\\
& \quad +d_{y}
(k_{D}(x,
y))_{(z,w)}(0,v)\\
& = d_{x}(k_{D}(x,y))_{(z,w)}[(d\rho_{
\varphi})_{z}(u),0]\\
& \quad +d_{y}(k_{D}(x,y))_{(z,w)}[0,(d\rho_{
\varphi})_{w}(v)]\\
& = d_{x}(k_{D}(x,y))_{(z,w)}d\varphi_{\xi}[(d\tilde{\rho}_{
\varphi})_{z}(u),]\\
& \quad +d_{y}(k_{D}(x,y))_{(z,w)}[0,d\varphi_{\eta}(d\tilde{\rho}_
{\varphi})_{w}(v)]\\
& = (dk_{\Delta})_{(\xi,\eta)}((d\tilde{\rho}_{
\varphi})_{z}
(u) , (d\tilde{\rho}_{\varphi})_{w}(v)).
\end{split}
\end{equation*}
\end{proof}

We are ready for the main theorem of this section: it provides us with a link
between an infinitesimal generator in \(D\) and its projections in the
images of complex geodesics.
It was proved in \cite[Proposition ~4.5]{bracci3}, but here we give a proof
independent of potential theoretic notions, along the same lines.
\begin{theorem}\label{teo2} Let \(D\) be a bounded, smooth, strongly
convex domain in \(\mathbb{C}^{n}\) and \(F:D\rightarrow\mathbb{C}^{n}\) a
holomorphic map. Then
\(F \in \textup{\textit{HolG}}(D)\) if and only if for any complex geodesics
\(\varphi:\Delta\rightarrow D\) the vector field on \(\Delta\)
given by \((d\tilde{\rho}_{\varphi})_{\xi}(F(\varphi(\xi)))\) is in
\(\text{HolG}(\Delta)\).
\end{theorem}

\begin{proof} Consider a complex geodesic \(\varphi\) and set
\(f_{\varphi}(\xi):=(d\tilde{\rho}_{\varphi})_{\varphi(\xi)} (F(\varphi(\xi))\).
So we have to prove that \(F\in \text{\textit{HolG}}(D)\) if and only if
\(f_{\varphi}\in
\text{\textit{HolG}}(\Delta)\) for all complex geodesics \(\varphi\). From the
previous
lemma it follows that \begin{multline*}
(dk_{D})_{(\varphi(\xi),\varphi(\eta))}(F(\varphi(\xi)),F(\varphi(\eta)))\\
\begin{aligned}& =(dk_{\Delta })_{ (\xi ,
\eta)}((d\tilde{\rho}_{\varphi})_{\varphi(\xi)}(F(\varphi(\xi))),(d\tilde{\rho}_
{\varphi})_{\varphi(\eta)}(F(\varphi(\eta))))\\
&
=(dk_{\Delta})_{(\xi,\eta)}(f_{\varphi}(\xi),f_{\varphi}(\eta)),\end{aligned}
\end{multline*} for any \(\xi\), \(\eta\) in \(\Delta\). Now on one hand, if
\(\varphi\) is any
complex geodesics and \(F\in \text{\textit{HolG}}(D)\) then \(f_{\varphi}\in
\text{\textit{HolG}}(\Delta)\) by by Theorem \ref{prop1}. On the other hand for
any \(z,w\) in \(D\) exists a unique complex geodesics \(\varphi
:\Delta \rightarrow D \) such that \(\varphi(\xi)=z\) and
\(\varphi(\eta)=w\) for some \(\xi, \eta \in \Delta \). Thus we have the
desired conclusion
\end{proof}

\section{Rigidity Condition}\label{sezione3}
In this section we are going to prove the main theorem of our paper.
Before going further we need a preparatory lemma. In \cite[Theorem
~6.3]{bracci1} it is proved that horospheres of radius \(R\) and fixed
center \(p\) are \(R\)-sublevel set of the global smooth function
\(\Omega_{\tau,p}\), which is the pluricomplex Poisson kernel defined on
\(D\). Here we do not need any particular property of \(\Omega_{\tau,p}\) other
than its smoothness and strict pseudo-convexity of its sublevel sets, and the
following
\begin{lemma}\label{lemma2}Let \(\tau \in \partial D \), \(p \in
D\) and let \(\Omega_{\tau,p} :D \rightarrow (0,+\infty)\) be the smooth convex
function, defined above, whose sublevel sets are \(\{x\in D
\mid
\Omega_{\tau,p}(x)« r\}=E(\tau,p,R)\), the
horospheres of center \(\tau\) and pole \(p\), for \(r » 0 \). Furthermore let
\(z\) be any point in \(\partial E(\tau,p,R)\) and let \(\varphi\) be a complex
geodesic
such that \(z,\tau \in
\overline{\varphi(\Delta)}\) and let
\((\varphi,\rho_{\varphi},\tilde{\rho}_{\varphi})\) be its projection device.
Then 
\begin{equation}\label{eq4}
\ker (d\rho_{\varphi})_{z}=T^{\mathbb{C}}_{z}\partial E(\tau,p,R)=\ker
(\partial \Omega_{\tau,p})_{z}
\end{equation}
where \(R=\Omega_{\tau,p}(z)\).\end{lemma} \begin{proof}

We have
\begin{equation*}\begin{split}
\lim_{y \rightarrow \tau}[k_{D}(\rho_{\varphi}(w),y)-k_{D}(p,y)]& =
\lim_{\varphi(\Delta) \ni y \rightarrow
\tau}[k_{D}(\rho_{\varphi}(w),y)-k_{D}(p,y)]\\
&\leq \lim_{\varphi(\Delta) \ni y \rightarrow
\tau}[k_{D}(w,y)-k_{D}(p,y)]\\ 
& =\lim_{y \rightarrow
\tau}[k_{D}(w,y)-k_{D}(p,y)],\end{split}\end{equation*} for any \(w\) in
\(D\). Thus observing that \(\rho_{\varphi}\circ \iota_{\varphi(\Delta)}=
id_{\varphi(\Delta)}\), we have \(\rho_{\varphi}(E(\tau,p,R)) =
E(\tau,p,R) \cap \varphi(\Delta)\) for every \(R»0\). Hence no point of the
fiber \(\{w\in D\mid \rho_{\varphi}(w)=z\}\) can lie inside \(E(\tau,p,R)\).
Since \(\rho_{\varphi} \) is holomorphic we have
\(d\rho_{\varphi}=\partial\rho_{\varphi}\). The fibers of \(\rho_{\varphi}\) are
intersections with \(D\) of complex affine
hyperplanes and we have
\begin{equation*}\ker
(d\rho_{\varphi})_{z}=\ker(\partial\rho_{\varphi})_{z}=\{w-z \mid w\in
\rho_{\varphi}^{-1}(z)\}.\end{equation*}
As we remarked in Section \ref{sezione1} (see also \cite{bracci1},
\cite{bracci2}),
horospheres in a bounded smooth strongly convex domain \(D\) are bounded smooth
convex subdomain of \(D\); so each point in the boundary of a horosphere has a
unique real
tangent hyperplane and also a unique complex tangent hyperplane. Hence the
complex hyperplane \(\{w\in \mathbb{C}^{n} \mid \rho_{\varphi}(w)=z\}\) is the
complex tangent hyperplane at \(\partial E(\tau,p,R)\) in \(z\) and
thus we can conclude that \begin{equation*}
\ker (d\rho_{\varphi})_{z}=T^{\mathbb{C}}_{z}\partial E(\tau,p,R)=\ker
(\partial \Omega_{\tau,p})_{z}.\end{equation*}
\end{proof}
We can now state and prove the following;\begin{theorem}\label{theo3} Let \(F\in
\text{HolG}(D)\), \(\tau \in \partial D\); furthermore for each projection
device
\((\varphi,\rho_{\varphi},\tilde{\rho}_{\varphi})\), let
\(f_{\varphi}(\xi):=(d\tilde{\rho}_{\varphi})_{\varphi(\xi)} (F(\varphi
(\xi))\). Then \(F\equiv 0 \) in \(D\) if and only if \( f_{\varphi}\equiv 0 \)
in \(\Delta\) for any complex
geodesic \(\varphi\) such that \(\tau \in \varphi(
\partial\Delta)\).
\end{theorem}
\begin{proof} The (\(\Rightarrow\)) direction is trivial, so suppose that
\(f_{\varphi}\equiv 0 \) in \(\Delta\) for every complex
geodesic \(\varphi\) such that \(\tau \in \varphi(\partial\Delta)\). Let
\(z\in D\) , \(p\in D\); let \(E(\tau,p,R)\) a horosphere such that
\(z\in \partial E(\tau,p,R)\) and \(\varphi\)
be
a complex geodesics such that
\(z=\varphi(\xi)\) for some \(\xi\in \Delta\) and \(\tau \in
\varphi(\partial\Delta)\). Thus we have
\begin{equation*}
(d\rho_{\varphi})_{z}(F(z))=(d\varphi)_{\xi}(d\tilde{\rho}_{\varphi})_{
\varphi(\xi) } (F(\varphi(\xi))=(d\varphi)_{\xi} f_{\varphi}(\xi)=0
\end{equation*} Thus from the identity (\ref{eq4}) it follows
that \((\partial\Omega_{\tau,p})_{z}F(z)=0\). Deriving with respect to
\(\overline{\partial}\) we have
\((\partial\overline{\partial}\Omega_{\tau,p})_{z}F(z)=0\) because
\(\overline{\partial}_{z}F(z)=0\) being \(F\) holomorphic. Now from the
previous Lemma \ref{lemma2}, we know that \(F(z)\in
T^{\mathbb{C}}_{z}\partial
E(\tau,p,R)\) where \(R=\Omega_{\tau,p}(z)\). By \cite[Remark
~6.4]{bracci1} boundaries of horospheres are
strictly pseudoconvex; so the
Levi form \(\partial\overline{\partial}\Omega_{\tau,p}\) restricted to
their complex tangent spaces has to be positive definite. So we must have
\(F(z)=0\). The arbitrariness of \(z \in D\) implies \(F\equiv 0 \) in \(D\) and
this concludes the proof.
\end{proof}

Now as an application of the previous result we can prove;
\begin{theorem}\label{theo4} Let
\(D\subset\subset\mathbb{C}^{n}\) be smooth, strongly
convex, \(\tau\in\partial D\), \(F\in \text{HolG}(D)\). If\begin{equation}\label{eq17}\angle\lim_{z\rightarrow
\tau}\dfrac{F(z)}{\lVert z-\tau
\rVert ^{3}}=0.\end{equation}
then \(F\equiv 0\) in \(D\).\end{theorem}
\begin{proof} We use the same
notation introduced in Theorem \ref{theo3}.
Let \((\varphi,\rho_{\varphi},\tilde{\rho}_{\varphi})\) be
a projction device at \(\tau\) such that \(\varphi(1)=\tau\). Let
\(\gamma (t):[0,1)\rightarrow \Delta\) any non-tangential 1-curve and \(\beta
(t)=\varphi (\gamma (t))\);
clearly \(\beta\) is a non tangential \(\tau\)-curve since
\(\varphi (\overline{\Delta})\) is transversal to the boundary of \(D\). From the existence of the non-tangential limit in (\ref{eq17}) and since
\((d\tilde{\rho}_{\varphi})_{z}\) depends continuously up to the boundary on
\(z\in
D\), we have
\begin{equation*}
\lim_{t\rightarrow
1}\dfrac{(d\tilde{\rho}_{\varphi})_{\beta (t)}F(\beta (t))}{
\lVert \beta (t) -\tau \rVert^{3}}=0.
\end{equation*}
Since \(f_{\varphi}(\xi):=d(\tilde{\rho}_{\varphi})_{\varphi(\xi)}
(F(\varphi(\xi)))\), we have
\(f_{\varphi}(t):=d(\tilde{\rho}_{\varphi})_{\beta (t)}
(F(\beta(t))\) and thus\begin{equation*}
\lim_{t\rightarrow
1}\dfrac{f_{\varphi}(\gamma (t))}{
\lVert \beta (t)-\tau \rVert^{3}}=0 .\end{equation*}So we can
conclude that \begin{equation*} \lim_{t\rightarrow
1}\dfrac{f_{\varphi}(\gamma (t))}{
\vert \gamma (t)-1\vert^{3}}=\lim_{t\rightarrow
1}\dfrac{f_{\varphi}(\gamma (t))}{
\lVert \beta (t)-\tau \rVert^{3}}\times \dfrac{\lVert
\beta(t)-\tau \rVert^{3}}{\vert \gamma (t)- 1\vert^{3}}=0
\end{equation*} because the second ratio in the product remains bounded in a
neighborhood of \(1\) being \(\varphi\) \(\mathcal{C}^{\infty}\) up to the
boundary.
Since \(\gamma\) is any arbitrary non tangential \(1\)-curve in \(\Delta\), the
previous limits implies also the angular limit in \(\Delta\)
\begin{equation*} \angle\lim_{\xi\rightarrow
1}\dfrac{f_{\varphi}(\xi)}{\vert\xi- 1\vert^{3}}=0.\end{equation*} By
Theorem \ref{teo2} \(f_{\varphi}\) is in \(\text{\textit{HolG}}(\Delta)\) and
now using Theorem \ref{shoresult1a} (but see \cite[Proposition ~2]{shoikhet1}
also) we have \(f_{\varphi}\equiv 0\)
in \(\Delta\). The previous Theorem \ref{theo3} and the arbitrariness of
\(\varphi\) allows us to conclude that \(F\equiv 0\) on \(D\).
\end{proof}

\bibliographystyle{amsplain}
\bibliography{bibliografiaaggiornata}

\end{document}